\documentclass[11pt]{amsart}
\usepackage{amscd}                                             \usepackage[arrow,matrix]{xy}
\usepackage{graphicx}
\usepackage{hyperref}
\usepackage{comment}
\usepackage{color}
\usepackage{amsmath, latexsym, amssymb}
\input xypic
\numberwithin{equation}{section}
\theoremstyle{plain}
\newtheorem{lemma}{Lemma}[section]
\newtheorem{proposition}[lemma]{Proposition}
\newtheorem{theorem}[lemma]{Theorem}
\newtheorem{corollary}[lemma]{Corollary}

\theoremstyle{definition}
\newtheorem{definition}[lemma]{Definition}
\newtheorem{remark}[lemma]{Remark}
\newtheorem{example}[lemma]{Example}

\begin{document}
\newcommand{\R}{{\mathbb R}}
\newcommand{\C}{{\mathbb C}}
\newcommand{\F}{{\mathbb F}}
\renewcommand{\O}{{\mathbb O}}
\newcommand{\Z}{{\mathbb Z}} 
\newcommand{\N}{{\mathbb N}}
\newcommand{\Q}{{\mathbb Q}}
\renewcommand{\H}{{\mathbb H}}

\newcommand{\dass}{da\ss~}

\newcommand{\Aa}{{\mathcal A}}
\newcommand{\Bb}{{\mathcal B}}
\newcommand{\Cc}{{\mathcal C}}    
\newcommand{\Dd}{{\mathcal D}}
\newcommand{\Ee}{{\mathcal E}}
\newcommand{\Ff}{{\mathcal F}}
\newcommand{\Gg}{{\mathcal G}}    
\newcommand{\Hh}{{\mathcal H}}
\newcommand{\Kk}{{\mathcal K}}
\newcommand{\Ii}{{\mathcal I}}
\newcommand{\Jj}{{\mathcal J}}
\newcommand{\Ll}{{\mathcal L}}    
\newcommand{\Mm}{{\mathcal M}}    
\newcommand{\Nn}{{\mathcal N}}
\newcommand{\Oo}{{\mathcal O}}
\newcommand{\Pp}{{\mathcal P}}
\newcommand{\Qq}{{\mathcal Q}}
\newcommand{\Rr}{{\mathcal R}}
\newcommand{\Ss}{{\mathcal S}}
\newcommand{\Tt}{{\mathcal T}}
\newcommand{\Uu}{{\mathcal U}}
\newcommand{\Vv}{{\mathcal V}}
\newcommand{\Ww}{{\mathcal W}}
\newcommand{\Xx}{{\mathcal X}}
\newcommand{\Yy}{{\mathcal Y}}
\newcommand{\Zz}{{\mathcal Z}}

\newcommand{\zt}{{\tilde z}}
\newcommand{\xt}{{\tilde x}}
\newcommand{\Ht}{\widetilde{H}}
\newcommand{\ut}{{\tilde u}}
\newcommand{\Mt}{{\widetilde M}}
\newcommand{\Llt}{{\widetilde{\mathcal L}}}
\newcommand{\yt}{{\tilde y}}
\newcommand{\vt}{{\tilde v}}
\newcommand{\Ppt}{{\widetilde{\mathcal P}}}
\newcommand{\bp }{{\bar \partial}} 

\newcommand{\Remark}{{\it Remark}}
\newcommand{\Proof}{{\it Proof}}
\newcommand{\ad}{{\rm ad}}
\newcommand{\Om}{{\Omega}}
\newcommand{\om}{{\omega}}
\newcommand{\eps}{{\varepsilon}}
\newcommand{\Di}{{\rm Diff}}
\newcommand{\vol}{{\rm vol}}
\newcommand{\Pro}[1]{\noindent {\bf Proposition #1}}
\newcommand{\Thm}[1]{\noindent {\bf Theorem #1}}
\newcommand{\Lem}[1]{\noindent {\bf Lemma #1 }}
\newcommand{\An}[1]{\noindent {\bf Anmerkung #1}}
\newcommand{\Kor}[1]{\noindent {\bf Korollar #1}}
\newcommand{\Satz}[1]{\noindent {\bf Satz #1}}

\renewcommand{\a}{{\mathfrak a}}
\renewcommand{\b}{{\mathfrak b}}
\newcommand{\e}{{\mathfrak e}}
\renewcommand{\k}{{\mathfrak k}}
\newcommand{\pg}{{\mathfrak p}}
\newcommand{\g}{{\mathfrak g}}
\newcommand{\gl}{{\mathfrak gl}}
\newcommand{\h}{{\mathfrak h}}
\renewcommand{\l}{{\mathfrak l}}
\newcommand{\sm}{{\mathfrak m}}
\newcommand{\n}{{\mathfrak n}}
\newcommand{\s}{{\mathfrak s}}
\renewcommand{\o}{{\mathfrak o}}
\newcommand{\so}{{\mathfrak so}}
\renewcommand{\u}{{\mathfrak u}}
\newcommand{\su}{{\mathfrak su}}
\newcommand{\ssl}{{\mathfrak sl}}
\newcommand{\ssp}{{\mathfrak sp}}
\renewcommand{\t}{{\mathfrak t }}
\newcommand{\Cinf}{C^{\infty}}
\newcommand{\la}{\langle}
\newcommand{\ra}{\rangle}
\newcommand{\half}{\scriptstyle\frac{1}{2}}
\newcommand{\p}{{\partial}}
\newcommand{\notsub}{\not\subset}
\newcommand{\iI}{{I}}               
\newcommand{\bI}{{\partial I}}      
\newcommand{\LRA}{\Longrightarrow}
\newcommand{\LLA}{\Longleftarrow}
\newcommand{\lra}{\longrightarrow}
\newcommand{\LLR}{\Longleftrightarrow}
\newcommand{\lla}{\longleftarrow}
\newcommand{\INTO}{\hookrightarrow}

\newcommand{\QED}{\hfill$\Box$\medskip}
\newcommand{\UuU}{\Upsilon _{\delta}(H_0) \times \Uu _{\delta} (J_0)}
\newcommand{\bm}{\boldmath}

\title[The  uniqueness of the Fisher metric]{\large The  uniqueness of the  Fisher metric  as information  metric}
\author[ H.V. L\^e ]{H\^ong V\^an L\^e }
 \date{\today}
 \thanks {H.V.L. is partially supported by RVO: 67985840}

\medskip
\address{
Institute  of Mathematics of ASCR,
Zitna 25, 11567  Praha 1, Czech Republic
}
\abstract
We define    a mixed topology on the fiber  space  $\cup_\mu \oplus^n L^n(\mu)$  over the space $\Mm(\Om)$ of all finite non-negative measures $\mu$ on  a   separable metric  space $\Om$ provided  with  Borel $\sigma$-algebra. We   define a   notion of strong continuity of a covariant $n$-tensor  field on $\Mm(\Om)$. Under  the assumption of strong continuity  of an information  metric  we prove  the uniqueness  of the Fisher metric  as information  metric  on  statistical models associated with $\Om$. Our  proof realizes  a    suggestion due to Amari and Nagaoka  to derive  the uniqueness of the Fisher metric  from  the   special case  proved by Chentsov by using  a special  kind of limiting procedure.  The obtained   result  extends   the   monotonicity characterization  of the Fisher metric on
statistical models associated with  finite  sample spaces  and  complement    the    uniqueness  theorem by  Ay-Jost-L\^e-Schwachh\"ofer that characterizes 
the  Fisher  metric by its  invariance  under sufficient statistics.   
\endabstract 


\maketitle
\tableofcontents

\section{Introduction}\label{sec:intro}

Recent successful  applications   of information  geometry, see e.g.  \cite{Amari1987, AN2000, AyJost, Shahshahani1979}, where the Fisher metric plays  a  fundamental role,   motivate us   to find an answer to the  following  important question.  Is there   another   metric on    statistical models  with  natural properties,  which   we could name  information metric?  

Intuitively,  information metric  should reflect the amount of non-negative information of a statistical model, moreover 
\begin{itemize}
\item   it should measure ``information loss"   associated with a data processing  and this information loss is a non-negative quantity   \cite[Axiom A]{Chentsov1978}; 
\item  it must be invariant under sufficient statistics, that is,  mappings between sample spaces that preserve all information  about the parameter $x$. 
\end{itemize}

In  statistical decision theory, a data  processing is a  statistical decision rule,  which can be deterministic  or randomized. A deterministic  decision rule is a  measurable map,  which is also called  a statistic. An indeterministic  decision rule is a Markov transition distribution \cite{Chentsov1982}. 
Recently,   Ay-Jost-L\^e-Schwachh\"ofer   showed  that a  transformation    between statistical models which is  associated with a Markov  transition distribution is a composition of  the inverse of  a transformation,  which is  associated with a sufficient statistic, and  a transformation which is associated  with a statistic  \cite[Theorem 4.10]{AJLS2012}. Hence, assuming  the condition of invariance  under sufficient   statistics,  the ``information loss" condition   is reduced to the case where  data processing  is associated with a statistic.

 Using the concept  of a  continuous local statistical  covariant tensor field   on statistical  models \cite[Definition 2.8]{AJLS2012}, see also Definition \ref{def:loc} below,  and utilizing the above discussion, we propose the following
 \
 
 \begin{definition}\label{def:information}  Given a class  $\{ \Om \}$ of measure spaces, {\it an information  metric}  on   statistical models (Definition \ref{def:gen}), or  more generally, on  parametrized  measure models   $(M, \Om, \mu, p)$ where $\Om \in \{ \Om \}$   is   a  continuous  local statistical non-negative definite quadratic  form $F_{(M, \Om, \mu, p)}$ (Definitions \ref{df:tensor}, \ref{def:new}, \ref{def:loc})   that
  satisfies the following  two conditions:
\begin{enumerate}
\item  the ``information loss"   $F_{(M, \Om, \mu, p)} - F_{(M, \Om_1, \kappa_*(\mu), \kappa_*(p))}$ is  a non-negative    definite quadratic form   for any statistic $\kappa: \Om \to \Om_1$;
\item the  ``information loss"  $F_{(M, \Om, \mu, p)} -F_{(M, \Om_1, \kappa_*(\mu), \kappa_*(p))}$ is zero (quadratic form)  if  $\kappa$ is sufficient  with respect to the parameter $x \in M$.
\end{enumerate}
\end{definition}

Each of the conditions (1) and (2)  in Definition \ref{def:information} is natural  and  has its own  appeal. The condition (2) has been   considered in    \cite{AJLS2012} as a  criterion for a  natural metric  on parametrized   measure models. The  condition (1)  is simpler  formulated than the condition (2), since   it does not depend on the  notion    of a sufficient  statistic, that  depends  on  a statistical model under consideration  and depends on the notion of information implicitly. 
(For a modern  definition of a  sufficient  statistic  we refer the reader to \cite{AJLS2013}, where Ay-Jost-L\^e-Schwachh\"ofer    propose  a geometric definition  of a  sufficient statistic   associated with a (signed) parametrized measure model in terms of  Banach   manifolds in consideration, which agrees with the old concept of sufficient statistics  that uses the Fisher-Neyman  characterization.)

In 1972 Chentsov  proved  that on   statistical models $(M, \Om, \mu, p)$ associated with  finite sample spaces $\Om$ 
the Fisher  metric  $g^F$ (Example \ref{ex:fish})  is a
unique metric, up to a multiplicative  constant,   that  satisfies  (2) \cite{Chentsov1982}. 
In \cite[Corollary 4.11]{AJLS2012}, for finite  sample  spaces $\Om$,  we derived  the  uniqueness (up to  a multiplicative constant)  of a metric that satisfies the condition  (1) on  statistical models  associated with  $\Om$  from the uniqueness  of a  metric  on statistical models that satisfies the condition (2) on $\Om$, see  Proposition \ref{cor:chentsov} and the Appendix at the  end of this note  for a discussion on the Chentsov theorem.
 The converse  statement,  every metric that satisfies  the condition (2)  also satisfies the condition (1),  follows  from the monotonicity  theorem  for the Fisher metric on statistical models  associated with finite sample spaces.


In 2012 Ay-Jost-L\^e-Schwachh\"ofer   proved  that  the Fisher  metric is  a  unique  metric,  up to  a multiplicative   constant, on statistical  models that  satisfies (2)  \cite[Remark 3.23]{AJLS2012}. (On   parametrized   measure  models  there  are  many  information metrics  that  satisfy the condition (2) \cite[Theorem 2.10]{AJLS2012}. This fact has been   observed  earlier   for      parametrized  measure models  associated   with    finite sample spaces  by Campbell in \cite{Campbell1986}). Further, Theorem 3.11  in \cite{AJLS2012}  states that, the Fisher  metric  satisfies  (1)  if  $\Om, \Om_1$ are smooth manifolds and $\mu$ is dominated by a Lebesgue measure.

\

In our  paper  we extend   the  aforementioned   results  as follows. Our first  observation is the following

\begin{theorem}\label{thm:main1}(The monotonicity of the Fisher metric) Let $\Om_1$, $\Om_2$  be  topological  spaces  with Borel $\sigma$-algebra, $\kappa: \Om_1 \to \Om_2$  a statistic.   Assume that  $(M, \Om_1, \mu_1, p_1)$   and $(M, \Om_2, \kappa_*(\mu_1), \kappa_*(p_1))$ are   $2$-integrable  parametrized measure models. Then    
for all $x \in  M$  and $V\in T_xM$   we have  $g ^F_{(M, \Om_1, \mu_1, p_1)}(V, V) \ge  g^F_{(M, \Om_2, \kappa_*(\mu_1), \kappa_*(p_1))}(V,V)$.  
\end{theorem}

Theorem \ref{thm:main1}    is possibly known  to experts in the field, but   we include  it  here   as well as its   short proof since we  have not seen  a  precise  statement with  a  proof of it  in   an available source  and we  wish to discuss   its consequence. 
We obtain  immediately  from the Ay-Jost-L\^e-Schwachh\"ofer theorem \cite[Remark 3.23]{AJLS2012} and  Theorem \ref{thm:main1} the  following

\begin{corollary}\label{cor:comp}  Let $\{ \Om\} $ be  the class  of    topological spaces provided with Borel $\sigma$-algebra.
Any    continuous local statistical non-negative definite quadratic  form $F$  on  statistical models associated with  $\{ \Om\}$   that satisfies  the condition (2) in Definition \ref{def:information}    also satisfies the condition (1) in Definition \ref{def:information}. In other  words, the condition (2)  is stronger  than the condition (1) for those $F$. 
\end{corollary}

To  prove the   uniqueness     result  for  an information metric that satisfies  the weaker monotonicity condition (1) in Definition \ref{def:information} we     pose  a topological condition on  such an information metric.  This condition  is formulated  in terms  of the    strong  continuity, the notion  we introduce  in Definition \ref{def:strong}.

\

For a measurable  space $(\Om, \Sigma)$  let us denote by $\Mm(\Om)$ the subset  of all finite non-negative measures  on $\Om$.

\begin{theorem}\label{thm:main2}  (The uniqueness  of the Fisher metric) Let $\{ \Om\} $ be  the class  of  separable metrizable   topological spaces provided with Borel $\sigma$-algebra. Assume  that $ F$  is a continuous local statistical non-negative definite quadratic  form  defined  on   all  2-integrable   statistical    models  $(M, \Om, \mu, p)$ (Definitions \ref{def:gen}, \ref{def:loc} )  where $\Om \in \{ \Om \}$.
 If  $F$  satisfies the  monotonicity condition (1) in Definition 1.1 and the  associated     quadratic form $\tilde F$  on $\Mm(\Om)$ (Definition \ref{def:loc}) is strongly continuous for all $\Om$, then    $F$   is the  Fisher quadratic form  up to  a multiplicative constant.
\end{theorem}

\begin{corollary}\label{cor:comp2} Let $\{ \Om\} $ be  the class  of  separable metrizable  topological spaces provided with Borel $\sigma$-algebra. 
Any    continuous local statistical non-negative definite quadratic  form $F$ on  statistical models associated with  $\{ \Om\}$ that satisfies  the condition (1) in Definition \ref{def:information}      also satisfies the condition (2) in Definition \ref{def:information} , if the associated  form $\tilde F$ on $\Mm(\Om)$  satisfies  the strong  continuity condition   for all $\Om$.
In other words,  the  combination  of  the condition (1)  and the strong continuity condition is stronger  than the   condition (2)  for  those  $F$. 
\end{corollary}


In Remark \ref{rem:positive}   below we    argue how   we   consider Theorem  \ref{thm:main2}       as a generalization of   the     characterization 
 the Fisher  metric  by its monotonicity  in the   case of finite sample  spaces, which is equivalent to the  Chentsov  theorem.  Since there are
many measure classes which are invariant under  statistics, see  e.g. \cite[Chapter 9]{Bogachev2007} for   discussion, we conjecture that  without the strong continuity assumption  there exists a  local     statistical  continuous metric  that satisfies (1)  but does not satisfy (2).

 
The remainder of our paper is   organized as follows.  In section \ref{sec:par} we   recall the notion  of  a $k$-integrable parametrized measure model  and the notion of 
a local statistical continuous covariant  tensor  field that  have been introduced by Ay-Jost-L\^e-Schwachh\"ofer in \cite{AJLS2012}. In section \ref{sec:mon} we  prove Theorem \ref{thm:main1}. 
In section \ref{sec:strong}  we   assume that $\Om$ is a separable metrizable topological space provided with  Borel $\sigma$-algebra.
We introduce  a mixed topology  on the space $\Ll^n_n(\Om) : = \cup_{\mu \in \Mm(\Om)}\oplus^n L^n(\Om, \mu)$, which enjoys   nice  properties (Proposition \ref{prop:top}). Using this topology we  introduce the notion  of strongly continuous covariant $n$-tensors  on $\Mm(\Om)$ (Definition \ref{def:strong}). 
In section \ref{sec:uni} we prove  Theorem \ref{thm:main2} by   deriving it  from the special case   associated  with  finite sample spaces.
Finally we include  an appendix containing  a  note on the Chentsov uniqueness theorem. 

The idea to derive   the uniqueness  of the Fisher metric  from its  special case     proved by Chentsov for  finite sample spaces     has been proposed by Amari and Nagaoka \cite[p. 39]{AN2000} as follows ``Here  we shall only observe that Chentsov's theorem  leads to the Fisher metric  and the $\alpha$-connections if a kind of limiting procedure is permitted", see  also Remark \ref{rem:positive} (3) on  a similar idea  due to Morozova-Chentsov.  In this  note  we  have found    such  limiting  procedure in  terms of strong continuity associated with the  mixed   topology. 


\section{$k$-integrable parametrized  measure  models and local statistical continuous tensor  fields}\label{sec:par}

For $\mu_0 \in \Mm(\Om)$  denote by 
$$\Mm_+(\Om, \mu_0) : = \{ \mu = \phi \mu_0|\, \phi \in L^1(\Om, \mu_0), \, \phi > 0 , \, \mu_0\text{-}a.e.\}, $$
$$ \Pp_+(\Om, \mu_0) =  \{\mu \in  \Mm_{+}(\Om, \mu_0) \; : \; \mu(\Omega)  = 1\}. $$


\begin{definition}\label{def:gen} (\cite[Definition 2.4]{AJLS2012}, cf. \cite[\S 2 , p. 25]{Amari1987}, \cite[\S 2.1]{AN2000}) Let $k \ge 1$.
A {\it $k$-integrable parametrized measure model } 
is a quadruple $(M, \Om, \mu, p)$ consisting  of a smooth (finite or infinite dimensional) Banach manifold 
$M$ and a continuous map $ p: M \to \Mm_+(\Om, \mu)$ provided with the $L^1$-topology such that  there exists a density potential $\bar p = \frac{dp}{d\mu}: M \times \Om \to \R$ satisfying $p(x) = \bar p(x,\om) d\mu(\om)$  and the following  conditions:
\begin{enumerate}
\item the function 
$x \mapsto \ln \bar p (x, \om) = \ln\frac{dp(x)}{d\mu}(\om): M \to \R$ is defined and continuously G\^ateaux-differentiable for $\mu$-almost all $\omega \in \Om$, 
\item   for all continuous vector field $V$ on $M$ the function $\omega \mapsto \p _{V} \ln \bar p(x, \om)$ belongs to $L^k(\Om, p(x))$;
moreover, the function $x \mapsto ||\p _{V} \ln \bar p(x, \om)|| _{L^k (\Om, p(x))}$  is continuous on $M$.
\end{enumerate}
We call $M$ the {\it parameter space} of $(M, \Om, \mu, p)$. 
We call  $(M, \Om, \mu, p)$ {\it  a statistical  model} if 
$p(M)\subset \Pp_+ (\Om, \mu)$.
\end{definition}

In Definition \ref{def:gen} the continuous G\^ateaux-differentiability of $\ln \bar p(x, \om)$  in $x \in M$ means the  continuity  of the  Gateaux-differential $D\ln \bar p(x, \om)$ as a function on $TM$ \cite[chapter I.3]{Hamilton1982}.
\begin{remark}\label{rem:log} In Definition \ref{def:gen}  we  represent  a  tangent vector $V \in T_xM$  by the function  $\p_V \ln \bar p (x, \om) \in L^1(\Om, p(x))$. This representation   is independent  of the choice  of a  reference  measure in $\Mm_+ (\Om,\mu)$, it depends only on the map $p : M \to  \Mm_+(\Om, \mu)$.
\end{remark}

\begin{definition}\label{def:new}(\cite[Definition 2.2]{AJLS2012}) 
A section  $\tau$  of the  bundle $T^*M \otimes _{n\,  times}\otimes  T^* M$  is  called {\it a  weakly continuous   covariant $n$-tensor},  if  
the value $\tau (V_n)$  is a continuous function for  any  continuous    $n$-vector field
$V_n$ on   $M$.
\end{definition}

\begin{definition} \label{df:tensor} (\cite[Definition 2.1]{AJLS2012})
A {\em covariant $n$-tensor  field }on $\Mm(\Om)$  assigns  
to each $\mu \in \Mm(\Om)$  a multilinear map $\tau_\mu: \bigoplus^n
L^n (\Om, \mu) \to \R$ that is continuous w.r.t. the product topology on $\bigoplus^n L^n (\Om, \mu)$.
\end{definition}

\begin{definition}[Locality and continuity condition] \label{def:loc}(\cite[Definition 2.8]{AJLS2012}) 
Given a class  $\{ \Om\}$ of  measure spaces, {\it a statistical} covariant continuous $n$-tensor field  $A$      
assigns  to  each  parametrized measure model $(M, \Om, \mu, p)$ where $\Om \in \{ \Om \}$
a   {\it weakly continuous} (in the sense of Definition \ref{def:new})  covariant $n$-tensor field $A|_{(M, \Om, \mu, p)}$ on $M$   (cf. Definition  \ref{df:tensor}).   
A statistical covariant   continuous  $n$-tensor field  $A$  is called {\it  local}  if  
there is  a  covariant $n$-tensor field
$\tilde A$ on $\Mm(\Om)$ with the following property  for any   parametrized measure model $(M, \Om, \mu, p)$  
and any $V _i \in T_xM$
\begin{equation}
A|_{(M, \Om, \mu, p)} (V_1, \cdots, V_n) = \tilde A_{p(x)}(\p _{V_1} \ln \bar p(x),  \cdots,\p_{V_n}\ln \bar p(x)) .\label{for:local}
\end{equation}
\end{definition}

From now on,  if a weakly continuous covariant  tensor  $A $ on  a $k$-integrable statistical model $(M, \Om, \mu, p)$   satisfies  (\ref{for:local})  for  $A|_{(M, \Om, \mu, p)} = A$, we shall write $A = p^*(\tilde A)$.

\begin{example}\label{ex:fish} (cf. Remark \ref{rem:fish}).  In  \cite{AJLS2012} Ay-Jost-L\^e-Schwachh\"ofer showed that
 {\it the Fisher quadratic form }
\begin{equation}
   g^F(V, W) _x  : = \int _\Om \p_V   \ln \bar p(x, \om) \p_W
    \ln\bar p(x, \om) \, dp(x)\label{for:fisher}
\end{equation} 
and   {\it the Amari-Chentsov  3-symmetric tensor}  
\begin{equation} 
   T^{AC} (V,W, X) _x : = \int _\Om \p_V  \ln\bar p(x, \om)
   \p_W  \ln\bar p (x, \om)\p_X  \ln\bar p(x, \om) \ d p(x)\label{for:amari}
\end{equation}
are local statistical continuous   covariant  tensor  fields.
\end{example}

\section{The monotonicity  of the Fisher metric}\label{sec:mon}

In this  section   we consider  topological spaces $\Om$  provided  with Borel $\sigma$-algebra.   We prove Theorem \ref{thm:main1}, 
and discuss   some related problems (Remark \ref{rem:suff}).

Recall that   a statistic  $\kappa: \Om_1 \to \Om_2$  induces the  linear operator $\kappa_*: L^1 (\Om_1, \mu_1) \to L^1 (\Om_2, \kappa_*(\mu_1))$    defined by \cite[(3.2)]{AJLS2012}
\begin{equation}
\kappa_*f (y) : = \frac{d\kappa_*(f \cdot  \mu_1)}{d\kappa_*(\mu_1)}(y) \label{eq:kappa}
\end{equation}
for $f \in L^1(\Om_1, \mu_1)$ and $y \in \Om_2$.

\begin{remark}\label{rem:meas} The  operator $\kappa_*$ is well defined,  since by the Radon-Nikodym theorem, $ f \in L^1 (\Om_1, \mu_1)$ if and only  if $f\cdot \mu_1$ is a measure  dominated by $\mu_1$, i.e. the null set  of $\mu_1$  is also a null set  of $f\cdot \mu_1$. Now  assume that $Z\subset \Om_2$ is a null set    of $\kappa_*(\mu_1)$. Then
$\kappa^{-1} (Z)$ is  also a null set  of $\mu$ and hence  of $f\cdot\mu_1$. It follows that $Z$ is a null set of $\kappa_*(f\cdot\mu_1)$, and by the Radon-Nykodym  theorem
$\kappa_*(f\cdot \mu_1)$ is dominated  by $\kappa_*(\mu_1)$.
\end{remark}

Some time  we  will write $\kappa^{\mu_1}_*(f)$,  if  $f$ may belong to $L^p(\Om_1,\mu_1)$  for different  $\mu_1$.


The following Lemma  \ref{prop:genf}   is an expression of the well-known fact that  condition expectation reduces   the $L^p$-norm, see  e.g. \cite[\S 4.3]{Neveu}. 

\begin{lemma}\label{prop:genf}For all $p\ge 2$ we  have $\kappa_*(L^ p(\Om_1, \mu_1)) \subset  L^p (\Om_2, \kappa_*(\mu_1))$.  The linear   map $\kappa_*$  contracts $L^p$-norm:  
$$||\kappa_*(f)||_{L^p(\Om_2, \kappa_*(\mu_1))} \le ||f||_{L^p(\Om_1, \mu_1)}$$
for all $f \in L^p(\Om_1, \mu_1)$.  
\end{lemma}

\begin{proof} Let $f \in L^p(\Om_1, \mu_1)$  and $y \in \Om_2$.  For a sequence  of  open sets $\Om_2= A_0 \supset \cdots \supset  A_n \supset \cdots  \ni y$ and a statistic $\kappa: \Om_1 \to \Om_2$ we set
$$ |f|_n (y): = \frac{ \int_{\kappa^{-1} (A_n)}|f (x)| \, d\mu_1} {\mu_1 (\kappa ^{-1}(A_n))}.$$
By the H\"older inequality, we have
$$(|f|_n(y)) ^p  \le \frac{\int_{\kappa^{-1} (A_n)}|f (x)^p| \, d\mu_1}{\mu_1 (\kappa ^{-1}(A_n))}.$$
 Since $\lim_{n \to \infty} |f|_n (y ) = \kappa_*(|f|)(y)$ we  deduce from the above inequality
\begin{equation}
(\kappa_*(|f|)(y))^p \le \kappa_* (|f | ^ p)(y).\label{eq:in1a}
\end{equation}
Using (\ref{eq:in1a}), we obtain
\begin{equation}
||\kappa_*(f)||_{L^p(\Om_2, \kappa_*(\mu_1))}^p  = \int_{\Om_2} |\kappa_*(f)|^p d\kappa_*(\mu_1) \nonumber 
\end{equation}
\begin{equation}
\le\int_{\Om_2} (\kappa_*(|f|))^p d\kappa_*(\mu_1) \le  \int_{\Om_2} \kappa_* (|f | ^ p)d\kappa_*(\mu_1)=  (||f||_{L^p(\Om_1, \mu_1)})^p, \nonumber 
\end{equation}
which implies immediately  Lemma \ref{prop:genf}.  
\end{proof}

{\it Proof  of Theorem \ref{thm:main1}}.  
By Remark \ref{rem:log}  the geometry  of a  parametrized measure model $(M, \Om_1, \mu_1, p_1)$ does not depend on the choice of a reference measure $\mu_1$. Thus, 
to prove Theorem \ref{thm:main1}   at a point $x \in M$, 
 we can assume that $p_1(x) = \mu_1$ and hence $\bar p_1 (x, \om) = 1$.  
Abusing the notation,  for   a  function  $\bar p : M \times \Om \to \R$  and for $x\in  M$, we  denote by $\bar p (x)$   the function $\Om \to \R$ such that
$\bar  p (x) (\om)  = \bar p (x, \om)$.  Then we  have   $ \kappa_*^{\mu_1}(\bar p_1(x))(\kappa(\om))  =1$ for  all $\om$.   Now let $V \in T_xM$.  Then  we have
\begin{equation}
\p_V (\ln \kappa_*^{\mu_1}(\bar p_1 (x))  = \p_V \kappa_*^{\mu_1}( \bar p_1 (x)).\label{eq:pv1}
\end{equation}
Next, we shall prove the  following   equality
\begin{equation}
 \p_V \kappa_*^{\mu_1}(\bar p_1  (x)) = \kappa_*^{\mu_1}(\p_V \bar p_1 (x)). \label{eq:pv2}
\end{equation}
To prove   (\ref{eq:pv2}) it suffices  to     show   that  the following equality holds
\begin{equation}
\p_V \kappa_*^{\mu_1} (p_1 (x)) = \kappa_*^{\mu_1} (\p _V  p_1 (x)) \label{eq:pv3}
\end{equation}
where  the RHS and LHS   of (\ref{eq:pv3}) are understood as signed measures.

The condition (2) in Definition \ref{def:gen}  implies that  $\p _V \bar p_1 \in L^2 (\Om_1, \mu_1)\subset L^1 (\Om_1, \mu_1)$, since $(M, \Om_1, \mu_1, p_1)$ is  a 2-integrable  parametrized  measure model.
Hence, for any measurable subset $A $ in $\Om_2$,  we can  apply the     differentiation under integral  (see e.g. \cite[Theorem  16.11, p. 213]{Jost2005})  to obtain  the following
$$\p _V  \int _{\kappa^{-1}(A)} \bar p_1 d\mu_1   = \int_{\kappa^{-1}(A)}\p_V \bar p_1 d\mu_1.$$
This  equality  implies  (\ref{eq:pv3})    immediately. Hence (\ref{eq:pv2}) holds.

Let us continue  the proof of  Theorem \ref{thm:main1}.  Using (\ref{eq:pv2}),  and recalling that $\bar p_1 (x, \om) = 1$, we obtain from (\ref{eq:pv1})
$$\p_V (\ln \kappa_*^{\mu_1}(\bar p_1 (x))  = \kappa_*^{\mu_1}(\p_V \bar p_1 (x))= \kappa_*^{\mu_1} (\p_V \ln \bar p_1)(x). $$
 Since $\p_V \ln \bar p_1(x) \in L^2(\Om_1, p_1(x))$  for all $x \in M$, by Lemma \ref{prop:genf}, we obtain 
\begin{equation}
||\kappa_*^{\mu_1}(\p_V \ln \bar p_1)(x)||_{L^2 (\Om_2, \kappa_*^{\mu_1}(p_1(x))} \le ||\p_V \ln \bar p_1(x)||_{L^2(\Om_1, p_1(x))}.\label{eq:mono}
\end{equation}
Noting that the LHS of (\ref{eq:mono})  is   equal  to $g^F_{(M, \Om_1, \mu_1, p_1)}$ and  the RHS of (\ref{eq:mono}) is equal to  $g^F_{(M, \Om_2, \kappa_*(\mu_1), \kappa_*(p_1))}$, we deduce  Theorem  \ref{thm:main1}  from (\ref{eq:mono}).

\begin{remark}\label{rem:suff} 1. It is not hard to see that if  $\Om_1$, $\Om_2$ are metric topological spaces, $\kappa$ and $f$  are continuous, then
the inequality (\ref{eq:mono}) becomes  an    equality   if and only if $f(\om) = \kappa_*(f) (\kappa(\om))$ for all $\om$.  

2.  Proposition \ref{prop:genf} implies  that  the  absolute value $\hat T^{AC}$  of the Amari-Chentsov tensor  defined by
 $\hat T^{AC}(V) : = |A^{TC} (V, V, V)| $  for $V \in  TM$
also satisfies      the version of  Definition \ref{def:information} on  statistical fields  which measure ``information loss".
\end{remark}

\section{Mixed topology and strongly continuous  covariant  tensor fields}\label{sec:strong}

In this section  we   assume that $\Om$ is a separable  metric topological space provided with  Borel $\sigma$-algebra.  Let $\R^n_{\ge 0}: = [0, \infty) ^n$. 
We introduce   a mixed topology on the spaces $\Ll^n_n(\Om) : = \cup_{\mu \in \Mm(\Om)}\oplus^n L^n(\Om, \mu)$  and $\Ll^n_1(\Om) : = \cup_{\mu \in \Mm(\Om)} L^n(\Om, \mu)$, which has  good properties (Proposition \ref{prop:top}). Using the mixed topology, we  introduce the notion  of strongly continuous covariant $n$-tensor fields  on $\Mm(\Om)$ (Definition \ref{def:strong}), whose examples are the Fisher quadratic form (Remark \ref{rem:fish})  and all  continuous functions on $\Ll^k_k(\Om_n)$ (Example \ref{exam:strong}), where $\Om_n$ is  a finite sample space consisting of  $n$ elementary events.  

\subsection{Mixed topology on $\Ll^n_n(\Om)$}\label{subs:mixed}
It  is known that  $\Mm(\Om)$ possesses  many different important topologies, e.g. the  total variation  topology, the strong topology and the weak topology. The  total variation is used in Definition \ref{def:gen}.  Now we recall the  notion  of  weak topology  on $\Mm(\Om)$, which plays  prominent  role in measure theory and  especially in  probability theory \cite{Bogachev2007, Bil1999}.   
Denote  by $C_b (\Om)$  the space of all  bounded  continuous real functions   on $\Om$. 

\begin{definition}\label{def:top} (cf. \cite[Definition 8.1.1, vol. II]{Bogachev2007})  
A sequence  of  Borel  measures $\mu_\alpha$  on $\Om$ is called {\it weakly convergent to}   a  Borel measure $\mu$  (writing  as $\mu_\alpha \LRA \mu$) if  for every  $f\in C_b(\Om)$ one has
$$\lim_\alpha  \int_\Om f \,d\mu _\alpha  = \int _\Om f d\mu.$$
\end{definition}

It is known that the weak topology on $\Mm(\Om)$ is generated by
fundamental neighborhoods  of $\mu$, $\mu \in \Mm(\Om)$, defined as follows \cite[Definition 8.1.2]{Bogachev2007}
\begin{equation}
U_{f_1,\cdots, f_k,\eps}(\mu) := \{ \nu: \,| \int_\Om f_i d\mu -\int_\Om f_i d\nu | < \eps \text{ for } i\in[1,k]\}, \label{eq:weak}
\end{equation}
where $f_i\in C_b(\Om)$, $k \in \N$ and $\eps > 0$. 


\begin{remark}\label{rem:cbd} 1. The weak topology  on $\Mm(\Om)$ is weaker than the total variation topology,  hence for any $k$-integrable  parametrized measure model  $(M, \Om, \mu, p)$ the  embedding $p: M \to \Mm_+ (\Om, \mu) \to  \Mm(\Om)$  is also  continuous with respect to the weak topology on $\Mm(\Om)$.

2. Since $\Om$ is a separable metric  topological space, for each $\mu \in \Mm(\Om)$  the subspace $ C_b(\Om)$ is a dense  subset  in $ L^n (\Om, \mu)$   with respect to the $L^n(\Om, \mu)$-topology  \cite{Adams2006, Jost2005, Bogachev2007}.
\end{remark}

Let us denote by $\Ll^n_n(\Om)$  the fibration  over $\Mm(\Om)$  whose fiber  over  $\mu\in \Mm(\Om)$ is the space $\oplus ^n L^n(\Om, \mu)$. Note that the   product topology on
$\oplus ^n L^n (\Om, \mu)$ is generated by the product norm  defined as follows. For $\vec{f} = (f_1, \cdots, f_n ) \in \oplus ^n L^n(\Om, \mu)$ let
$$||\vec{f}||_{L^n_n (\mu)} : = \sum_{i=1} ^n ||f_i||_{L^n  (\Om, \mu)}.$$
 Denote by
$\pi$ the  projection $\Ll^n_n(\Om)\to \Mm(\Om)$. 

 We  are going to define  a  topology on $\Ll^n_n(\Om)$ by specifying      its base.
For  an $n$-tuple  of functions $\vec{f} \in  \oplus^n C_b(\Om)= (C_b(\Om))^n$, an open  set $U\subset  \Mm(\Om)$ in the weak topology  and $\eps >0$ we set
\begin{equation}
O(\vec{f},U, \eps): =\{ [\vec{g}, \mu]:\: \mu \in U ,\, \vec{g} \in \oplus ^n L^n(\Om, \mu) \text {  and  }  ||\vec{g} - \vec{f}||_{L^n_n(\mu)}  < \eps\},\label{eq:mix}
\end{equation}
where $[\vec g, \mu]$ means a pair.

 Note that
\begin{equation}
O(\vec{f}, \cup_i U_i, \eps) = \cup_i O(\vec{f}, U_i, \eps) \text{ and } O(\vec{f},\cap _i U_i, \eps) =  \cap _i O(\vec{f}, U_i, \eps). \label{eq:uni}
\end{equation}

\begin{proposition}\label{prop:top}  The collection $B$ of all subsets  $O(\vec{f}, U, \eps)$ where  $\vec{f} \in (C_b(\Om))^n$, $U $ is open set in $\Mm(\Om)$ and
$\eps > 0$  generates a unique  topology   on $\Ll^n_n(\Om)$, which we shall call the mixed  topology. Furthermore, the restriction  of this topology to each  fiber 
$\oplus ^n L^n(\Om, \mu)$  is  equal to  the $L^n(\Om,\mu)$-topology
on the fiber. Consequently,  the space $(C_b(\Om) )^n\times \Mm(\Om)$
is a dense  subset  in the mixed   topology. The projection $\pi: \Ll^n_n(\Om) \to \Mm(\Om)$ is continuous with respect  to the mixed topology on the domain and the weak topology on the target space.
\end{proposition}

\begin{proof}  To prove the first assertion of Proposition \ref{prop:top} it suffices to  show that the following  conditions hold.
\begin{enumerate}
\item The (base) elements  in $B$ cover $\Ll^n_n(\Om)$.
\item Let $O(\vec{f_1}, U_1, \eps_1)$  and  $O(\vec{f_2}, U_2, \eps_2)$ be base elements. If their  intersection  $I$  is non-empty, then for each $[\vec{f}, \mu] \in I$, there is a base element $O(\vec{f_3}, U_3, \eps_3)$ such that $[\vec{f}, \mu]\in O(\vec{f_3}, U_3, \eps_3)\subset I$.
\end{enumerate}

The first condition (1)  holds  by Remark \ref{rem:cbd}.2.

Now let us prove that (2)  holds.  For $\vec{f} \in \oplus ^nL^n (\Om, \mu)$  we set
$$B(\vec{f},\eps, \mu): = \{ \vec{f'} \in \oplus ^nL^n(\Om, \mu) \text{ and } ||  \vec{f'} - \vec{f}||_{L^n_n (\Om, \mu)} < \eps\}.$$ 
Note that  $I \cap \pi^{-1}(\mu)$ is an open subset  of $\pi^{-1}(\mu)$ in $L^n (\Om,\mu)$-topology, since it is the intersection of two open balls $B(\vec{f_1},\eps_1, \mu)$ and $B(\vec{f_2},\eps_2,\mu)$.
Using(\ref{eq:uni}),  we can assume w.l.o.g.
$$U_1=U_{\tilde f_1, \cdots, \tilde f_k, \eps_1 }(\mu_1),$$
$$U_2= U_{\tilde g_1 \cdots, \tilde g_m, \eps_2}(\mu_2).$$
Let $\delta_1$ be a number  such that 
\begin{equation}
U_{\tilde f_1,\cdots , \tilde f_k,\tilde g_1, \cdots, \tilde g_m,\delta_1}(\mu) \subset U_1\cap U_2,\label{eq:inters1}
\end{equation}
and moreover $\delta _1 \le \min\{ 1, \eps_1, \eps_2\}$. Next we choose  a positive number $\delta _2 \le \delta _1$ such that
\begin{equation}
||\vec{f}-\vec{f_1} ||_{L^n_n(\mu)} < \eps_1-\delta_2 \text{ and } ||\vec{f}-\vec{f_2}||_{L^n_n(\mu)} < \eps_2 -\delta_2.\label{eq:delta} 
\end{equation}

Then we   choose $[\vec{f_3}, \mu] \in I\cap  \pi^{-1}(\mu)$ with the following properties
\begin{equation}
\vec{f_3} \in (C_b (\Om))^n  \text{ and }  ||\vec{f_3} - \vec{f}||_{L^n_n(\mu)} < \frac{1}{4} \delta_2. \label{eq:f3}
\end{equation}
We obtain from (\ref{eq:delta}) and (\ref{eq:f3})
\begin{equation}
||\vec{f_3} -\vec{f_i}||_{L^n_n(\mu)} < \eps_i- \frac{3}{4} \delta_2  \text{ for } i =1,2.\label{eq:delta3}
\end{equation}
We write $\vec{f_3} = (f_3^1, \cdots, f_3^n)$.  Note that $|f_3^i - f_1^i|^n$ and $|f_3^i-f_2^i|^n$ are continuous  bounded   functions on $\Om$  for all $ i \in [1,n]$.
Now we set
\begin{equation}
U_3 : =  U_{\tilde f_1, \cdots , \tilde f_k, \tilde g_1, \cdots, \tilde g_m, |f^i_3 - f^i_1|^n,|f^i_3 - f^i_2|^n, i\in [1, n], (\frac{1}{ 8} \delta _2)^n} (\mu).\label{eq:u3}
\end{equation}
Since $\delta_2 \le \delta _1$ we obtain  from  (\ref{eq:u3}) and (\ref{eq:inters1})
$$U_3 \subset U_{\tilde f_1, \cdots, \tilde f_n, \tilde g_1, \cdots, \tilde  g_m, \delta_1} (\mu) \subset U_1 \cap U_2.$$
Clearly, (\ref{eq:f3}) implies that $[\vec{f},\mu] \in O(\vec{f_3}, U_3, \frac{1}{4} \delta_2)$. Hence, setting  $\eps_3: = \frac{1}{ 4} \delta_2$,  to complete  the proof  
of the first assertion of Proposition \ref{prop:top}, it suffices to show that
\begin{equation}
O(\vec{f_3}, U_3, \frac{1}{ 4} \delta _2) \subset I.\label{eq:subset}
\end{equation}
Let $[\vec{h}, \mu'] \in O(\vec{f_3}, U_3, \frac{1}{ 4} \delta_2)$. To prove (\ref{eq:subset}) we need to show that $[\vec{h}, \mu'] \in I$,  or equivalently
\begin{equation}
[\vec{h}, \mu'] \in O(\vec{f_i}, U_i, \eps_i) \text{ for } i = 1, 2. \label{eq:subset2}
\end{equation}
Since $\mu'\in U_3 \subset U_i$ for $i = 1,2$, (\ref{eq:subset2})  is equivalent to
\begin{equation}
||\vec{h} -\vec{f_i}||_{L^n_n(\mu')} < \eps_i  \text{ for } i = 1,2.\label{eq:dist}
\end{equation}
Taking into account $[\vec{h}, \mu'] \in O(\vec{f_3}, U_3, \frac{1}{ 4} \delta_2)$,   we obtain
\begin{equation}
||\vec{h}-\vec{f_3}||_{L^n_n(\mu')} < {1\over 4} \delta_2.\label{eq:in1}
\end{equation}
Since $\mu'\in U_3$,  we derive from  (\ref{eq:delta3})  and (\ref{eq:u3})
\begin{equation}
||\vec{f_3} -\vec{f_1}||_{L^n_n(\mu')} <||\vec{f_3} -\vec{f_1}||_{L^n_n(\mu)} +{1\over 8} \delta_2 < \eps_1 -\frac{5}{ 8} \delta_2.\label{eq:in2}
\end{equation}
In the same way we obtain
\begin{equation}
||\vec{f_3} -\vec{f_2}||_{L^n_n(\mu')} < \eps_2 -\frac{5}{ 8} \delta_2. \label{eq:in3}
\end{equation}
Clearly, (\ref{eq:in1}), (\ref{eq:in2}),  and (\ref{eq:in3}) imply (\ref{eq:dist}). This proves the first assertion of  Proposition \ref{prop:top}.

The second  assertion  of Proposition \ref{prop:top}  follows  from Remark \ref{rem:cbd}.2, observing that a ball $B(\vec{f}, \eps, \mu) $
is  the intersection of  the open set $O(\vec{f}, U(\mu), \eps)$ with the  fiber $\oplus^n L^n(\Om, \mu)$. 
  
 Finally, the last  assertion is obvious, since the  preimage $\pi^{-1} (U)$ of  an open set $U \subset \Mm(\Om)$ is the union of all
 open sets  of the form $O(\vec{f}, U, \eps)$, $f \in (C_b(\Om))^n$ and $\eps > 0$. This completes the proof of Proposition \ref{prop:top}.
\end{proof}

\subsection{Strongly continuous   covariant $n$-tensor on $\Mm(\Om)$}\label{subs:scont}

\begin{definition}\label{def:strong} A covariant  $n$-tensor field on $\Mm(\Om)$ is called  {\it strongly continuous}, if it  is a continuous  function on $\Ll^n_n(\Om)$ with respect  to the  mixed   topology.
\end{definition}

\begin{example}\label{exam:strong}  Let $\Om_n: = \{ \om_1, \cdots, \om_n\}$ be a    finite   sample space of $n$ elementary events. Let $\delta _{\om_i}$ denote the Dirac measure concentrated at $\om_i$.
Let  $\mu_l = \sum _{ i=1} ^l c_i \delta _{\om_i} \in \Mm (\Om_n)$, where $l \le n$ and $c_i > 0$. Then, for all $k \ge 1$, $L^k(\Om_n, \mu_l)$ is homeomorphic to $C_b(\Om_l) = \R^l$, which is
provided with  the usual  (vector space) topology.  Furthermore,  the weak topology on  $\Mm(\Om_n)= \R^n_{\ge 0}$   coincides with the usual topology on $\R^n_{\ge 0}\subset \R^n$.  Hence the    subset $\Mm_+ (\Om_n)$ consisting  of  positive measures on $\Om_n$  is dense  in $\Mm(\Om_n)$. We observe that  $\pi: \Ll_k ^k (\Om_n) \to \Mm(\Om_n)$ is a fiber  bundle whose
fiber over  $\mu_l$  is homeomorphic to $(\R ^l) ^k$.   A covariant $k$-tensor
field $\tilde F$ on $\Mm(\Om_n) = \R^ n _{\ge 0}$  is a continuous  function  on $\Ll_k ^k (\Om_n)$. 
Since   $\pi^{-1} (\Mm_+ (\Om_n))$  is open and dense  in $\Ll _k ^ k (\Om_n)$,    the  function  $\tilde F$ is defined  uniquely by its restriction  to $\pi^{-1} (\Mm_+(\Om_n))$.  In particular,  the Fisher metric  defined  on $\Mm_+(\Om_n)$ is associated with   the quadratic form
$\tilde  g^F : \Ll_2^2(\Om_n)\to \R$  defined by $\tilde  g^F ([f_1, f_2, \mu]) = \int_{\Om_n} f_1 \cdot f_2\, d \mu$, see also Remark \ref{rem:fish}. 
\end{example}

\begin{proposition}\label{prop:fisher}   
Let $g\in C_b (\Om)$  and $ c: \Mm(\Om) \to \R$  be a continuous function  with respect to the weak topology. We define  a   covariant $n$-tensor field $T_{(g,c)}$ on $\Mm(\Om)$ by  setting
$$T_{g, c} ([  f_1, \cdots, f_n, \mu]) :=  c(\mu) \cdot \int_{\Om} g \cdot f_1\cdots  f_n\,d\mu.$$
Then $T_{g,c}$  is a  strongly  continuous covariant  $n$-tensor  field on $\Mm(\Om)$. 
\end{proposition}
\begin{proof} By  Proposition \ref{prop:top}, $\pi: \Ll^n_n(\Om) \to \Mm(\Om)$  is   a continuous  function,  hence $c(\mu)$ is a continuous function on
$\Ll^n_n(\Om)$. Thus to prove   Proposition  \ref{prop:fisher} it suffices to assume that $c(\mu) = 1$,  i.e.  it suffices to show that $T_{g, 1}$ descends to  a   continuous function on $\Ll^n_n (\Om)$  provided  with the  mixed topology.
Equivalently, we need  to show  that the set 
$$ O(a, b) : = \{ [\vec{f}, \mu] \in  \Ll^n_n (\Om) |\:  a< T_{g, 1} (\vec{f}, \mu) < b \} $$
is an open set in the mixed topology for any  $-\infty < a < b < \infty $.

Let $[\vec{f}, \mu] \in O(a+\eps,b-\eps)$, where $\eps < {1\over  4} (b-a)$.  We  will show that  there is an open set $O(\vec{f_1}, U_1, \delta) \ni  [\vec{f}, \mu]$ such that
\begin{equation}
T_{g,1} (O(\vec{f_1}, U_1, \delta)) \subset  (a, b).\label{eq:cont}
\end{equation}

\begin{lemma}\label{lem:est1}
The restriction of $T_{g,1}$ to each fiber $\oplus ^n L^n (\Om,\mu)$  is continuous in the  product   $L^n(\Om, \mu)$-topology.  Moreover,
if $||\vec{h}-\vec{f}||_{L^n_n (\mu)} \le 1$  then  
 $$|T_{g,1} ([\vec{f}, \mu]) - T^n_{g,1} ([\vec{h}, \mu])| \le  \sup_\Om g(\om) \cdot 2 ^n \cdot ||\vec{h} -\vec{f}||_{L^n _n(\mu) }  \cdot(1+\sum _{i =1}^n \sum_{ k  = 1} ^n  ||f_i||_{L^n (\Om,\mu)}^k). $$
\end{lemma}
\begin{proof}
Write  $\vec{f}-\vec{h} = \vec{a} = (a_1, \cdots, a_n)$.  Expanding $h_1 \cdots  h_n = \Pi_{i =1} ^n (f_i-a_i)$  and using  Holder's inequality, we obtain
\begin{eqnarray}
|T_{g,1}([\vec{f}, \mu]) -T_{g,1}([\vec{h}, \mu])|\le  \nonumber \\
\sup_\Om g(\om) \cdot \sum_{ [1, n] = \{i_1, \cdots, i_k\} \cup \{j_1, \cdots, j _{n -k}\} } \int_\Om| a_{i_1} \cdots  a_{i_k} f_{j_1} \cdots  f_{j_{n-k}}| d\mu\le \nonumber \\
\sup_\Om g(\om)\cdot 2 ^n\cdot \max_{ [1, n] = \{i_1, \cdots, i_k\} \cup \{j_1, \cdots, j _{n -k}\} } ||a_{i_1}||_{L^n (\Om, \mu)} \cdots ||f_{j_{n-k}}||_{L^n (\Om, \mu)}.\nonumber\\
 \label{eq:est1}
\end{eqnarray}
Note that  in  (\ref{eq:est1}) the  set $\{j_1, \cdots, j _{n -k}\}$    may be  empty  but   the set  $\{i_1, \cdots, i_k\}$ is always non-empty. Since $\sum_{i=1}^n ||a_i||_{L^n (\Om, \mu)} \le 1$, we have 
\begin{eqnarray}
\max_{ [1, n] = \{i_1, \cdots, i_k\} \cup \{j_1, \cdots, j _{n -k}\} } ||a_{i_1}||_{L^n (\Om, \mu)} \cdots ||f_{j_{n-k}}||_{L^n (\Om, \mu)} \le\nonumber \\
 \sum_{i=1}^n ||a_i||_{ L^n (\Om, \mu)} ( 1 + \sum _{i =1}^n \sum_{ k  = 1} ^n  ||f_i||_{L^n (\Om,\mu)}^k).\label{eq:est1a}
\end{eqnarray}
Clearly Lemma \ref{lem:est1}  follows from  (\ref{eq:est1})  and (\ref{eq:est1a}).
\end{proof}

We define a function $G : \Ll^n_n (\Om)  \to \R$ by setting
$$G([\vec{f}, \mu]) : =\sup_\Om g(\om)\cdot  2^n ( 1+ \sum _{i =1}^n \sum_{ k  = 1} ^n  ||f_i||_{L^n (\Om,\mu)}^k). $$



Let  us pick   an element $\vec{f_1}= ((f_1)_1, \cdots  ,(f_1)_n)\in (C_b(\Om) )^ n \cap B (\vec{f}, \delta, \mu)$ where $\delta$ is so small such that  the following  equalities hold:
\begin{equation}
\delta <\min \{ \frac{ 1}{2}, \frac{\eps} {16G([\vec{f}, \mu]))}\}, \label{eq:delta1}
\end{equation}
\begin{equation}
|T_{g, 1} (\vec{f_1}, \mu) - T_{g,1} (\vec{f}, \mu) | < {\eps\over 16}, \label{eq:eps}
\end{equation}
and 
\begin{equation}
|G([\vec{h}, \mu])- G([\vec{f_1}, \mu])| \le {\eps \over 8 } \label{eq:es1a}
\end{equation}
for all $h  \in B(\vec{f_1}, \delta, \mu)$.
The existence   of $\delta$ follows  from the  positivity of $G$,  from Lemma \ref{lem:est1} and from the continuity  of  the restriction   of $G$  to each fiber $\oplus ^n L^n (\Om,\mu)$.   

We  define a neighborhood $U_1  =U_1([\vec{f_1}, \mu])$   containing $\mu$ as follows
$$U_1: = U_{(g\cdot f_1 \cdots  f_n),   |(f_1)_1|^n,  \cdots , |(f_1)_n|^n),\,\lambda}(\mu),$$
where $\lambda $  depends on $g, \vec{f_1}, \mu$   and is so small  such that
\begin{equation}
\lambda < \frac{\eps}{8}   \label{eq:lambda1}
\end{equation} 
 and for $\mu' \in U_1$   we have
\begin{equation}
|G([\vec{f_1}, \mu']) -G([\vec{f_1}, \mu])| \le {\eps\over 8} .\label{eq:norm}
\end{equation}
The existence  of $\lambda$ satisfying (\ref{eq:norm}) is ensured by the  continuity  of  the function $G([\vec{f_1}, \mu])$ in variable $\mu$.

Now we shall show   that $O(\vec{f_1}, U_1, \delta) \ni [\vec{f}, \mu]$  satisfies  (\ref{eq:cont}). 
Assume that $[\vec{h}, \mu'] \in O(\vec{f_1}, U_1, \delta)$.  Then
\begin{eqnarray}
|T_{g,1} ([\vec{h},\mu']) - T_{g,1}([\vec{f}, \mu])|  \le  |T_{g,1} ([\vec{h},\mu']) - T_{g,1}([\vec{f_1}, \mu'])|\nonumber \\
 + |T_{g,1}([\vec{f_1}, \mu']) - T_{g,1}([\vec{f_1}, \mu])| +  |T_{g,1}([\vec{f_1}, \mu])-  T_{g,1}([\vec{f}, \mu])|.\label{eq:est2}
\end{eqnarray}
Let us estimate the first term in  the RHS of (\ref{eq:est2}). By  Lemma \ref{lem:est1}  we have
\begin{equation}
|T_{g,1} ([\vec{h}, \mu']) - T_{g,1} ([\vec{f_1}, \mu'])| \le ||\vec{h} - \vec{f_1}||_{L_n^n (\mu')} \cdot G([\vec{f_1}, \mu']).\label{eq:2a}
\end{equation}
 Taking into account (\ref{eq:norm}), (\ref{eq:es1a}),  and the choice  of $\delta$ in (\ref{eq:delta1}),  we obtain from  (\ref{eq:2a}), noting that $ \vec{f_1} \in B(\vec{f}, \delta, \mu) \LRA  \vec{f}\in B(\vec{f_1},\delta, \mu)$:
 \begin{eqnarray}
|T_{g,1} ([\vec{h}, \mu']) - T_{g,1} ([\vec{f_1}, \mu'])|\le \delta \cdot G([\vec{f_1}, \mu'])\le\nonumber \\
  \delta ( \frac{\eps}{ 8} + G([\vec{f_1}, \mu]))\le \delta ( \frac{\eps}{ 8}  + \frac{\eps}{8} +  G([\vec{f}, \mu])< \frac{3\eps }{16}.\label{eq:2b}
 \end{eqnarray}
 We estimate  the  second  term in the RHS of (\ref{eq:est2}) as follows, using  the fact   $\mu' \in U_1= U_1([\vec{f_1}, \mu])$ with $\lambda$ satisfying (\ref{eq:lambda1}):
\begin{equation}
 |T_{g,1}([\vec{f_1}, \mu']) - T_{g,1}([\vec{f_1}, \mu])| <  \lambda < \frac{\eps}{8}. \label{eq:est2term}
 \end{equation}
 Using (\ref{eq:2b}),(\ref{eq:est2term})  and  estimating the last term in the RHS of (\ref{eq:est2})  by (\ref{eq:eps}),  we obtain from  (\ref{eq:est2})
\begin{equation}
|T_{g,1} ([\vec{h},\mu']) - T_{g,1}([\vec{f}, \mu])|  \le  \frac{3\eps }{16} + \frac{\eps}{8} + \frac{\eps}{16} = \frac{3\eps}{8}.\label{eq:est3}
\end{equation}
 (\ref{eq:est3}) implies that  $T_{g,1} ([\vec{h},\mu']) \in (a, b)$. Hence  (\ref{eq:cont}) holds.
The proof   of  Proposition  \ref{prop:fisher} is completed.
\end{proof}

\begin{remark}\label{rem:fish} Let $[1]: \Om \to \R$ denote   the constant  function  taking the  value 1. Then $[1]\in C_b (\Om)$.  Let
$(M, \Om, \mu, p)$  be a 2-integrable  parametrized  measure model. By (\ref{for:local}) the 2-tensor field $T_{[1], 1}$ induces  the following
local statistical   2-tensor $g$ on $(M, \Om, \mu, p)$:
\begin{eqnarray}
g_x (V, W) = (T_{[1], 1})_{p(x)} (\p_V \ln \bar p(x), \p_W \ln \bar p(x))\nonumber\\
 = \int_\Om \p_V \ln \bar p(x)\cdot \p_W \ln \bar p(x) \, dp(x).\label{eq:t11}
 \end{eqnarray}
The RHS  of (\ref{eq:t11})  is the Fisher metric $g^F$. Thus,   the Fisher  metric  is  induced  from the strongly continuous covariant 2-tensor field $T_{[1],1}$  on $\Mm(\Om)$. In the same way,  the  Amari-Chentsov  tensor
$T^{AC}$  is induced  from the   strongly continuous covariant 3-tensor field  $T_{[1], 1}$ on $\Mm(\Om)$.
\end{remark}

\section{The uniqueness  of  the Fisher metric}\label{sec:uni}

Recall that $\delta _\om$  denotes the Dirac  measure concentrated at $\om \in \Om$. 

\begin{lemma}\label{lem:dirac}  (cf. \cite[Example 8.1.6]{Bogachev2007}).  The  set   of  all measures  of  the form $\sum_{i=1}^N  c_i \delta_{\om_i}$,
$c_i >0$, is dense  in $\Mm(\Om)$  in the weak topology.  The convex hull of  the  set of Dirac  measures is dense in the space $\Pp(\Om)$.
\end{lemma}
\begin{proof} 1.  A version  of  Lemma \ref{lem:dirac}  for  finite  Baire measures   is proved  in \cite[Example 8.1.6]{Bogachev2007}.  We apply Bogachev's   argument  for  the proof  of Lemma \ref{lem:dirac}.  Suppose we are given a neighborhood $U\ni \mu$ of the form (\ref{eq:weak}).
We may assume that  the total  variation norm $||\mu||\le 1$. There are simple (step) functions $g_i$ such that
$\sup_{\om\in \Om} |f_i (\om) - g_i (\om)| < \eps/4$ for all $i \in [ 1,k]$. To prove Lemma \ref{lem:dirac} it  suffices to  show that $U$ contains a
measure $\nu =\sum_{i=1}^N  c_i \delta_{\om_i}$  such that  for all $i \in [1, k]$  we have
\begin{equation}
 \int_\Om g_i d\mu = \int_\Om  g_i  d\nu. \label{eq:dirac}
 \end{equation}
Let  $\Om = \cup_{j=1}^{n_i}  A_i^j$ be  a  finite  partition  into
disjoint measurable sets    corresponding  to $g_i$, i.e. $g_i = \sum a_i ^ j \chi_ {A_i^j}$.  
Then
$$ \Om = \cup_{l_1, \cdots, l_k} A_{1} ^{l_1} \cap  A_2^{l_2} \cap \cdots  \cap A_{k} ^{l_k} $$
is  a finite  partition  corresponding  to $g_i$ for all $ i \in [1, k]$.  
Set $c_{l_1 \cdots l_k} : = \mu(A_{1} ^{l_1} \cap  A_2^{l_2} \cap \cdots  \cap A_{k} ^{l_k})$  and  let $\om_{l_1 \cdots l_k}$  be  a point in $A_{1} ^{l_1} \cap  A_2^{l_2} \cap \cdots  \cap A_{k} ^{l_k}$.  Then (\ref{eq:dirac}) holds
for $\nu = \sum_{l_1, \cdots ,l_k} c_{l_1 \cdots l_k}  \delta_{\om_{l_1 \cdots l_k}}$. Since $\mu$ is a non-negative measure,  we have $c_{l_1 \cdots l_k} \ge  0$.
This completes the proof of  the first assertion of  Lemma \ref{lem:dirac}.  

2. The second   assertion follows  immediately, since by the above construction  of  
$\sum_{i=1}^N  c_i \delta_{\om_i}$ we have $\sum c_i = \mu(\Om)$.
\end{proof}

{\it  Proof  of Theorem \ref{thm:main2}}.   Assume that    $F$ is  a metric  defined  on all 2-integrable  statistical models $(M, \Om, \mu, p)$  that satisfies  the condition of  Theorem 1.4  and $\tilde F_\Om$  denotes the associated  strongly continuous    quadratic form on $\Mm(\Om)$.  Denote by $\tilde g ^F _\Om$  the quadratic  form on $\Mm(\Om)$ that is associated with  the Fisher metric $g^F$.  We shall show that  $\tilde F_\Om =  c\cdot  \tilde g ^ F_\Om$  for some  constant $c$.  

  By Proposition \ref{cor:chentsov} it suffices to   consider the case   $\Om$ is non-discrete. 
Let  $\kappa_n: \Om \to \Om_n$ be a statistic such that  $\kappa_n (\Om) = \Om_n$.   Let  us choose    points $\{\om_1, \cdots, \om _n \}\in \Om$   such that $\kappa_n (\om_i)$  are distinct
points in $\Om_n$.  Let us consider the following map
$$\Om_n \stackrel{i_n}{\to } \Om \stackrel{\kappa_n}{\to} \Om_n,$$
where $i_n $ identifies $\kappa _n (\om_i)$ with  $\om_i$  for  all $i \in [1, n]$.  Note that $i_n$ is also a statistic  and  $\kappa_n \circ i_n = Id$. Let $\mu_n ^+ \in \Pp_+ (\Om_n)$.   
Observe that $(\Pp_+ (\Om_n), \Om_n, \mu_n^+, Id)$ is a   2-integrable statistical model. By the monotonicity assumption  of $F$, and using  $\kappa _n \circ i_n = Id$,  we conclude that
the    metric $F$ defined  on the 2-integrable  statistical model  $(\Pp_+ (\Om_n), \Om, (i_n)_*(\mu_n^+), (i_n)_*(Id))$  is   defined  uniquely  by the metric $F$  defined on
the 2-integrable  statistical model\\
 $(\Pp_+ (\Om_n), \Om_n, \mu_n^+, Id)$.   By Proposition \ref{cor:chentsov} the   metric $F$  defined  on the  2-integrable  statistical model $(\Pp_+ (\Om_n), \Om_n, \mu_n ^+, Id)$ coincides  with the Fisher metric up  to   a multiplicative constant  $c$.   Hence,  the   restriction  of $\tilde F_\Om$ to the subspace of $ \Ll_2 ^2 (\Om)$
$$\Ll_2 ^2(\om_1, \cdots, \om_n):  \{[f_1, f_2, \mu_n]\in \Ll_2^2 (\Om)|\,  \mu_n  = \sum_{i=1}^n c_i \delta _{\om_i},  c_i > 0 \}$$
coincides with  the restriction of $\tilde g^F_\Om$  up to the multiplicative constant $c$,  since $\tilde F_\Om$ is strongly continuous.

Now we shall  show that  the constant  $c$  does  not depend  on the choice  of   a collection $\{\om_1, \cdots, \om_n\}$.
Let $\{\om_1 ', \cdots, \om_m'\}$ be another  collection of  distinct $m$ points on $\Om$. Let $\Om_N : =\{ \om_1'', \cdots  , \om _N ''\}$
be  the union of $\{\om_1, \cdots, \om_n \}$  and  $\{\om_1 ', \cdots, \om_m'\}$. 
We consider  the following    sequence  of  statistics
$$\Om_n\stackrel{i_{n, N}}{\to }\Om_N \stackrel{i_N}{\to} \Om\stackrel{\kappa_N}{\to} \Om_N \stackrel{\kappa_{N, n}}{\to}\Om_n,$$
where $i_{n, N}$ and  $i_{N}$ are the natural embeddings  and $\kappa_N$ and $\kappa_{N, n}$  are  sufficient  statistics  such that  
$\kappa_N \circ i_N = Id$ and $\kappa_{N, n}  \circ  i _{n, N} = Id$.
By Proposition \ref{cor:chentsov},  the constant $c$ that depends on $\{ \om_1, \cdots, \om_n\}$   equals  the constant $c ''$ that depends  on $\{\om_1'', \cdots  , \om _N ''\}$.  In the  same way
we  prove  that  the constant $c'$ that depends on $\{\om_1 ', \cdots, \om_m'\}$  equals the constant $c ''$ that depends  on $\{\om_1'', \cdots  , \om _N ''\}$. Hence  the constant  $c$  does not depend  on the  choice  of $\{\om_1, \cdots, \om_n\}$.

We  denote by $\Dd^+ (\Om )$ the set  of all measures  $\mu_n=  \sum_{i=1}^n c_i \delta _{\om_i},  c_i > 0 $, where $\om_i \in \Om$.
 By Lemma \ref{lem:dirac} the  subset 
$$\Ll^2_2 (\Om, \Dd^+) : = \{[f_1, f_2, \mu]\in \Ll^2_2(\Om)| \, \mu \in \Dd^+(\Om)\}$$
 is dense  in $\Ll^2_2(\Om)$ in the   mixed topology.   Since the restriction  of $\tilde F_\Om$ to $\Ll^2_2 (\Om, \Dd^+)$ coincides with the restriction of $\tilde g^F_\Om$ up to the multiplicative  constant $c$,
taking into account the   strong continuity of $\tilde F_\Om$, this completes the proof of  Theorem \ref{thm:main2}.\qed

\

\section{Appendix:  The Chentsov  uniqueness  theorem}
In this  Appendix  we  recall  a  reformulation  of the Chentsov  theorem  \cite[Theorem 11.1, p. 159]{Chentsov1978}  on the uniqueness of the Fisher metric in the  language of information geometry  by  Amari and Nagaoka (Proposition \ref{prop:chentsova}), which is  simpler  than the original  formulation by Chentsov in  the category language. In  Proposition \ref{cor:chentsov}
we  formulate  a result  in \cite{AJLS2012} that  characterizes the Fisher metric on finite  sample spaces via the monotonicity.  Then we  
 discuss in Remark \ref{rem:positive} some  problems  in generalizing the  Chentsov theorem   to     a larger class of  measure  spaces   that  contains also   non-discrete measure spaces.

 

Let us denote by $\Pp_+ (\Om_n)$   the subset  of $\Pp(\Om_n)$  that consists of positive  measures.

\begin{proposition}\label{prop:chentsova} 
(\cite [Theorem 2.6, p. 38]{AN2000}) Assume that $\{(g_n)\}_{n=1}^\infty$ is  a sequence of Riemannian metrics on $\Pp_+ (\Om_n)$  for each $n$ that are invariant with respect to sufficient 
statistics; i.e., for all  $n, m, S \subset \Pp_+(\Om_n)$, and $F : \Om_n  \to \Om_m$ such that  $F$ is a sufficient 
statistic for $S$, the induced metrics  on $S$ and $S_F$ are assumed to be invariant. Then there exists a positive real number $c$  
such that, for all $n$, $g_n$ coincides with the Fisher metric on $\Pp_ +(\Om_n)$ scaled by a factor of $c$. 
\end{proposition}

Amari and Nagaoka did not    supply  their  proof  of  Proposition  \ref{prop:chentsova}. We 
 recommend  the reader to    \cite{Campbell1986}  for a  slight generalization  of the Chentsov theorem, whose proof is close  to the  original Chentsov's proof.
For the reader convenience we recall   the  following monotonicity characterization of the Fisher    metric  on finite sample spaces.

\begin{proposition}\label{cor:chentsov}(\cite[Corollary 4.11]{AJLS2012})  Let $F$ be a continuous local statistical quadratic  2-form  defined  on statistical  models  associated with finite sample spaces $\{\Om_n\}$  such that $F$ is monotone  under  sufficient  statistics.
Then $F$  coincides  with the Fisher metric   up to a  multiplicative constant.
\end{proposition}

\begin{remark}\label{rem:positive} (1)  Chentsov  defined
the  Fisher  metric  only on  the positive  sector  $\Pp_+(\Om_n)$
of the  space of all  probability  measures because  the expression  for the Fisher metric in (\ref{for:fisher})  is well-defined  only  on $\Pp_+(\Om_n)$. In this paper     we follow
the approach  in \cite{AJLS2012} by requiring that    an information metric  $F$ is obtained by  (\ref{for:local}) from
the associated  2-form $\tilde  F$,  which is not only defined  on
$\Pp_+(\Om_n)$   but  also   defined    on $\Mm (\Om_n)$ (in general case, on $\Mm(\Om)$)  and hence  on $\Pp (\Om_n)$ (resp.  on $\Pp (\Om)$).   This  small  difference  is important, since   for  a non-discrete   space $\Om$  we do not know
how  to define  a  notion of a positive measure without  using a reference  measure $\mu_0$. Since the Fisher metric  $g^F$ satisfies  the mentioned   requirement, see Example \ref{exam:strong},  Proposition \ref{cor:chentsov}
is  equivalent  to the Chentsov  uniqueness theorem. Clearly, Theorem \ref{thm:main2}   generalizes  Proposition \ref{cor:chentsov}.  

(2) As  we   mentioned above, the original Chentsov theorem can be 
equivalently reformulated  in terms  of  the  associated    form $\tilde F$.  Note  that  the space  $\Pp(\Om_n)$  (resp.  $\Mm(\Om_n)$) is not a manifold,  or a manifold  with boundary, but a stratified space  which admits   different embeddings  into   Euclidean  spaces. 
In  \cite{AJLS2013}  and in the  present paper   we  do not  consider     smooth  tensor fields on  $\Pp(\Om_n) $  (resp. on $\Mm(\Om_n)$)  but    (strongly or point-wise) continuous       tensor fields on $\Mm(\Om)$ which  do not  require the   notion of a  smooth structure  on $\Mm(\Om)$.

(3) In  \cite[\S 5]{MC1991}  Morozova-Chentsov  also suggested   a method   to  extend   the   Chentsov uniqueness  theorem  to  the case of 
non-discrete  measure spaces $\Om$. Their idea   is similar to the Amari-Nagaoka idea, namely they wanted to consider  a Riemannian metric on  infinite  measure spaces as limit   of  Riemannian metrics on finite measure spaces. 
They did not  discuss a condition  under which   such a  limit  exists.  In fact,  they did not give a definition   of limit  of   such  metrics.   
If the limit exists     they   called   it {\it finitely generated}.  They stated  that the Fisher  metric  is the unique   finitely generated metric  that is invariant  under  sufficient statistics  
(resp.  that is monotone).
One may speculate that since  such a   Riemanian metric  depends    on  base measures $\mu$ and  tangent  vectors  at $\mu$  Morozova-Chentsov's approach  requires   a definition of topology   on the space $\Ll_2 ^2 (\Om)$.   

\end{remark}

\section*{Acknowledgement} 
The author   thanks  Shun-ichi Amari, Nihat Ay,   Lorenz Schwachh\"ofer  and Alesha  Tuzhilin for valuable conversations. She is grateful to  Vladimir Bogachev  and J\"urgen Jost for their helpful comments and suggestions.  The final version   of this manuscript is   greatly improved  thanks  to  critical  helpful  suggestions of the anonymous  referees. She  acknowledges the  VNU for Sciences in Hanoi for excellent working conditions  and financial support during  her visit when a part  of this note has been done.

\end{document}